\documentclass[11pt]{amsart}
\usepackage[foot]{amsaddr}
\usepackage{color}
\usepackage{amsmath,amsfonts}
\usepackage{amsfonts,amsmath,amsthm,amssymb,latexsym}
\usepackage{mathrsfs}

\usepackage{algorithm}
\usepackage{algorithmic}
\usepackage{ulem}

\usepackage{hyperref}

\newtheorem{theorem}{Theorem}[section]
\newtheorem{proposition}[theorem]{Proposition}
\newtheorem{lemma}[theorem]{Lemma}
\newtheorem{corollary}[theorem]{Corollary}
\theoremstyle{definition}
\newtheorem{definition}[theorem]{Definition}
\newtheorem{example}[theorem]{Example}

\theoremstyle{remark}
\newtheorem{remark}[theorem]{Remark}
\numberwithin{equation}{section}
\theoremstyle{remark}

\DeclareMathOperator{\lv}{lv}
\DeclareMathOperator{\init}{init}
\DeclareMathOperator{\sepa}{sep}

\DeclareMathOperator{\Sol}{Sol}
\newcommand{\Pu}{\langle \langle x \rangle \rangle}
\newcommand{\C}{\mathbb{C}}

\newcommand{\N}{\mathbb{N}}

\newcommand{\V}{\mathbb{V}}
\newcommand{\Va}{\mathcal{V}}

\newcommand{\sys}{\mathcal{S}}

\newcommand{\qq}{\mathbf{p}_0}

\newcommand{\Res}{\mathrm{res}}

\newcommand{\field}{K}
\newcommand{\exfield}{L}
\newcommand{\acfield}{\overline{\exfield}}
\newcommand{\indets}{\mathrm{Coord}}
\newcommand{\leadvars}{\mathrm{LV}}

\allowdisplaybreaks

\begin{document}

\date{\today}
\title{Algebraic and Puiseux series solutions of systems of autonomous algebraic ODEs of dimension one in several variables}

\author{Jos\'{e} Cano}
\address{Dpto. Algebra, an\'alisis Matem\'atico, Geometr\'{\i}a y Topolog\'{\i}a, Universidad de Valladolid, Spain.}
\email{jcano@agt.uva.es}

\author{Sebastian Falkensteiner}
\address{Research Institute for Symbolic Computation (RISC), Johannes Kepler University Linz, Austria.}
\email{falkensteiner@risc.jku.at}

\author{Daniel Robertz}
\address{Lehrstuhl f\"ur Algebra und Zahlentheorie, RWTH Aachen University, Germany.
}
\email{daniel.robertz@rwth-aachen.de}

\author{J.Rafael Sendra}
\address{Universidad de Alcalá,
	Dpto. Física y Matemáticas,
	Alcalá de Henares,
	Madrid,
	Spain}
\email{rafael.sendra@uah.es}

\thanks{
	First author partially supported by PID2019-105621GB-I00 from the Spanish MICINN. Second and fourth authors partially supported by the grant PID2020-113192GB-I00 (Mathematical Visualization: Foundations, Algorithms and Applications) from the Spanish MICINN. Second author also supported by the Austrian Science Fund (FWF): P 31327-N32.
}

\begin{abstract}
In this paper we study systems of autonomous algebraic ODEs in several differential indeterminates. 
We develop a notion of algebraic dimension of such systems by considering them as algebraic systems. 
Afterwards we apply differential elimination and analyze the behavior of the dimension in the resulting Thomas decomposition. 
For such systems of algebraic dimension one, we show that all formal Puiseux series solutions can be approximated up to an arbitrary order by convergent solutions.
We show that the existence of Puiseux series and algebraic solutions can be decided algorithmically.
Moreover, we present a symbolic algorithm to compute all algebraic solutions. 
The output can either be represented by triangular systems or by their minimal polynomials.
\end{abstract}

\maketitle


\section{Introduction}
In~\cite{cano2019existence}, we have studied local solutions of first order autonomous algebraic ordinary differential equations (shortly AODEs).
Therein we prove that every fractional power series solution of such equations is convergent, and an algorithm for computing these solutions is provided.
In~\cite{cano2020algebraic}, we generalized these results to the case of systems of higher order autonomous AODEs in one unknown function whose associated algebraic set is of dimension one, i.e.\ the algebraic set is a finite union of curves and, maybe, points.
Here, in the current paper, we show that every component of a fractional power series solution vector of a system of higher order autonomous AODEs in several unknown functions, with the expected dimension, is convergent.
In~\cite{Denef1984} it is shown that for general systems of AODEs the existence of non-constant formal power series solutions can not be decided algorithmically.
Nevertheless, for the systems treated in the paper and formal Puiseux series solutions, this undecidability property does not hold. 
Moreover, all algebraic solutions can be computed algorithmically.

We follow the idea in~\cite{cano2020algebraic} and triangularize the given system and perform differential elimination. 
In contrast to~\cite{cano2020algebraic}, we use in the current paper the so-called Thomas decomposition (see e.g.~\cite{robertz2014formal,bachler2012algorithmic}). 
In the literature there are several other methods to triangularize differential systems and to obtain resolvent representations of them, see for instance~\cite{CluzeauHubert2003,boulier1995representation} and references therein, but the Thomas decomposition seems to be the most adequate and direct method for our reasonings.

In the Thomas decomposition of the systems under investigation, we obtain subsystems involving the unknown functions and only one derivative of one of the functions. 
In the case where the first solution component of such a subsystem is an algebraic Puiseux series, the subsystem can be further simplified in order to obtain subsystems involving only the unknown functions. 
As a consequence, all solution components are algebraic and for every component of a possible solution its minimal polynomials can be found. 
In the representation in terms of the minimal polynomials, however, not all possible root combinations indeed define a solution of the given differential system. 
So we alternatively represent the solutions as the triangular set involving no derivatives. 

The main algorithm presented here generalizes both the algorithms in~\cite{sendra2015rational} and~\cite{cano2020algebraic}, in the following sense. 
In~\cite{sendra2015rational}, rational solutions of systems of autonomous AODEs of dimension one with one differential unknown are computed. 
Additionally, algebraic and formal Puiseux series solutions for such systems were considered in~\cite{cano2020algebraic}. 
We present here an extension to systems of this kind involving several differential unknowns.

The structure of the paper is as follows.
In Section~\ref{sec-pre} we recall some necessary concepts such as simple systems, the Thomas decomposition and algebraic solutions of first order autonomous AODEs. 
In Section~\ref{sec-algebraicGeometricDimension} we introduce, and briefly discuss, a notion of dimension when dealing with differential systems. 
Following our algebraic definition, which can be computed by just using well-known methods from computational Algebraic Geometry, the dimension does not increase when applying algebraic and differential elimination methods for obtaining a Thomas decomposition (see Theorem~\ref{thm:ThomasDimension}). 
Consequently, differential systems of algebraic dimension one decompose into simple subsystems of a very specific type, namely into triangular systems involving no derivatives of the unknown functions except for possibly the first derivative of one variable.
In Section~\ref{sec-systems} we study formal Puiseux series solutions of such systems with constant coefficients. 
The main result is Theorem~\ref{theorem-convergence2}, where it is shown that the Artin approximation theorem for algebraic systems (see~\cite{artin1968solutions}) holds for autonomous differential systems of algebraic dimension one. 
Additionally we show in Theorem~\ref{thm:existence} that the existence of constant and non-constant formal Puiseux series solutions can be decided.
In Section~\ref{sec-systemalgebraic} we focus on algebraic solutions.
In Theorem~\ref{thm:algSolDecomposition} is shown that all algebraic solutions of systems of autonomous AODEs of algebraic dimension one can be given by a finite number of simple systems involving no derivatives. 
Alternatively to the representation by simple systems, the algebraic solutions can be represented by their minimal polynomials and Proposition~\ref{prop:minpolynomialsystem} gives a degree bound on them. 
In Section~\ref{sec-Algorithms}, by using the algorithms in~\cite{cano2020algebraic,bachler2012algorithmic}, we derive an algorithm for determining all algebraic solutions of the original system. 
This is also illustrated by examples.

\section{Preliminaries}\label{sec-pre}
In this section we first recall the notion of simple systems and Thomas decomposition
in a form that is adapted to systems of ordinary differential equations.
For further details we refer to~\cite{robertz2014formal}.
Afterwards we recall some results on algebraic solutions of first order autonomous AODEs.
We follow the work~\cite{aroca2005algebraic}, where the authors give a necessary and sufficient condition on algebraic general solutions of such equations. 
They indicate how to use these results in order to compute all algebraic solutions. 
We need some results that are stated in~\cite{aroca2005algebraic} without an explicit proof.
For the sake of completeness, we present detailed proofs of them.

\subsection{Algebraic and Differential Thomas Decomposition}\label{subsec:Thomas}

Let $\field$ be a field of characteristic zero
and $\field[z_0,\ldots,z_m]$ the polynomial ring in $m+1$ indeterminates with ordering $z_0<\cdots<z_m$.
For $F \in \field[z_0,\ldots,z_m] \setminus \field$ we denote by $\lv(F)$ the \textit{leading variable}
of $F$ with respect to $<$.
An algebraic system is given by
\begin{equation}\label{eq-simplealgebraic}
\sys = \{F_1=0,\ldots,F_M=0, U_1 \ne 0,\ldots,U_N\ne 0 \},
\end{equation}
where $F_i$, $U_j \in \field[z_0, \ldots, z_m]$ are polynomials. 
Moreover, let $\exfield \supseteq \field$ be a field extension and let us denote by $\acfield$ the algebraic closure of $\exfield$.
For given $S \subset \field[z_0, \ldots, z_m]$, we introduce the set
$$\Va_{\acfield}(\sys) := \V_{\acfield}(\{ F_1,\ldots,F_M \}) \setminus \bigcup_{1 \le i \le N} \V_{\acfield}(\{ U_i \}) \subseteq {\acfield}^{m+1},$$
where $\V_{\acfield}(\Delta)$ is the zero set defined by the polynomials in $\Delta$. 
Each non-constant polynomial is considered recursively as a univariate polynomial
in its leading variable with coefficients
that are univariate polynomials in lower ranked indeterminates, etc. Correspondingly, a sequence
of projections is defined by
$$\begin{array}{ccccccccc}
\acfield^{m+1} \! & \! \stackrel{\pi_m}{\longrightarrow} \! & \!
\acfield^m \! & \! \stackrel{\pi_{m-1}}{\longrightarrow} \! & \!
\acfield^{m-1} \! & \! \stackrel{\pi_{m-2}}{\longrightarrow} \! & \! \ldots \! & \!
\stackrel{\pi_1}{\longrightarrow} \! & \! \acfield\\[0.2em]
(z_0, \ldots, z_m) \! & \! \longmapsto \! & \!
(z_0, \ldots, z_{m-1}) \! & \! \longmapsto \! & \!
(z_0, \ldots, z_{m-2}) \! & \! \longmapsto \! & \! \ldots \! & \!
\longmapsto \! & \! z_0
\end{array}$$
The \textit{initial} $\init(F)$ of a non-constant polynomial $F$ is the coefficient of the highest power of $\lv(F)$ in $F$,
which is a polynomial in indeterminates that are ranked lower than $\lv(F)$.
Similarly, the \textit{discriminant} of $F$ is also defined with respect to $\lv(F)$.

\begin{definition}\label{de:algsimple}
	The algebraic system~$\sys$ is said to be \textit{simple} (with respect to $<$) if
	\begin{enumerate}
		\item the polynomials $F_1,\ldots,F_M, U_1,\ldots,U_N$ are not in $\field$ and have
		pairwise distinct leading variables;
		\item the initial and discriminant of each of these polynomials (say, with leading variable $z_{j+1}$)
		do not admit roots $(a_0, \ldots, a_j) \in \acfield^{j+1}$ in
		$(\pi_{j+1} \circ \ldots \circ \pi_m)(\Va_{\acfield}(\sys))$.
	\end{enumerate} 
\end{definition}

An \textit{algebraic Thomas decomposition} of a system $\sys$ as in~\eqref{eq-simplealgebraic} is
a finite collection of (algebraic) simple systems $\sys_i$ such that
\begin{align*}
\Va_{\acfield}(\sys) = \dot\bigcup \, \Va_{\acfield}(\sys_i),
\end{align*}
where $\dot\cup$ denotes the disjoint union, see~\cite[Definition~2.2.10]{robertz2014formal}.

Now let $\field$ be a differential field of characteristic zero, 
endowed with the derivation $'={\rm d}/{\rm d}x$, and consider
the differential polynomial ring $\field\{ y_1, \ldots, y_n \}$ in
differential indeterminates $y_1$, \ldots, $y_n$ representing unknown functions of $x$.
In other words, we consider the polynomial ring
over $\field$ in $y_1$, \ldots, $y_n$ and all their derivatives $y'_1$, $y''_1$, \ldots,
$y'_n$, $y''_n$, \ldots, which are linearly ordered by the \textit{ranking}
\begin{equation}\label{eq:ranking}
y_1 < y'_1 < y''_1 < \ldots < y_2 < y'_2 < y''_2 < \ldots < y_n < y'_n < y''_n < \ldots
\end{equation}
(more generally, any ranking on the differential polynomial ring $\field\{ y_1, \ldots, y_n \}$ is assumed to satisfy $y^{(l)}_j < y^{(k)}_j$
whenever $k > l$.)
Then for any differential polynomial $F \in \field\{ y_1, \ldots, y_n \} \setminus \field$
its leading variable $\lv(F)$ (with respect to $<$) and its initial $\init(F)$
are defined as above. The \textit{separant} $\sepa(F)$ of $F$ is defined as the
partial derivative of $F$ with respect to $\lv(F)$, which is also the initial of any
proper derivative of $F$.

Let $F$, $G \in \field\{ y_1, \ldots, y_n \} \setminus \field$. Then $F$ can be reduced modulo $G$
if the leading variables of $F$ and $G$ involve the same differential indeterminate,
say, $\lv(F) = y^{(k)}_j$ and $\lv(G) = y^{(l)}_j$, and if $k > l$ or
$k = l$ and the degree $d_F=\deg(F,y^{(k)}_j)$ of $F$ in $y^{(k)}_j$ is greater than or equal to the
degree $d_G=\deg(G,y^{(k)}_j)$. In this case a (differential) pseudo-reduction
of $F$ modulo $G$ is possible as follows.
\begin{itemize}
	\item If $k = l$, then let
	\[
	R = \init(G) \, F - \init(F) \, (y^{(k)}_j)^{d_F - d_G} \, G.
	\]
	\item If $k > l$, then let
	\[
	R = \sepa(G) \, F - \init(F) \, (y^{(k)}_j)^{d_F - 1} G^{(k-l)}.
	\]
\end{itemize}
Then $R$ is either constant or has leading variable $\lv(R) = \lv(F)$ and degree in $\lv(F)$ less than $d_F$
or has leading variable $\lv(R)$ ranked lower than $\lv(F)$.
If none of the above reduction steps can be performed, then $F$
is said to be \textit{(differentially) reduced modulo $G$}.

Let $G_1$, \ldots, $G_r \in \field\{ y_1, \ldots, y_n \} \setminus \field$ (typically with
pairwise distinct leading variables). Iterated pseudo-reduction of $F$
modulo each of the polynomials $G_1$, \ldots, $G_r$ yields, after finitely many steps, a
differential polynomial $R$, called
\textit{(differential) pseudo-remainder} of $F$ modulo $G_1$, \ldots, $G_r$,
that is reduced modulo $G_1$, \ldots, $G_r$.
Taking the recursive representation
of (differential) polynomials into account, pseudo-reductions similar to the ones defined above
may be applied to ensure that all coefficients of $R$ and their coefficients, etc.,
are (differentially) reduced modulo $G_1$, \ldots, $G_r$ as well. If we would like to emphasize
this property, we call $R$ \textit{completely reduced modulo} $G_1$, \ldots, $G_r$.

A \textit{differential system} is given by
\begin{equation}\label{eq-simple}
\sys = \{F_1=0,\ldots,F_M=0, U_1 \ne 0,\ldots,U_N \ne 0\},
\end{equation}
where $F_1$, \ldots, $F_M$, $U_1$, \ldots, $U_N \in \field\{ y_1, \ldots, y_n \}$ are differential polynomials.

\begin{definition}\label{de:diffsimple}
The differential system $\sys$ is said to be \textit{simple} (with respect to $<$) if
\begin{enumerate}
	\item \label{diffsimple:1} $\sys$ is simple as an algebraic system;
	\item \label{diffsimple:2} each $F_i$ is reduced modulo $F_1$, \ldots, $F_{i-1}$, $F_{i+1}$, \ldots, $F_M$;
	\item \label{diffsimple:3} each $U_j$ is reduced modulo $F_1$, \ldots, $F_M$.
\end{enumerate}
If $F_i$ and $U_j$ in~\eqref{diffsimple:2} and~\eqref{diffsimple:3}, respectively, are completely reduced
modulo $F_1$, \ldots, $F_{i-1}$, $F_{i+1}$, \ldots, $F_M$ and modulo $F_1$, \ldots, $F_M$, respectively,
then we call $\sys$ a \textit{simple differential system with completely reduced equations and inequations}.
\end{definition}

Let $\exfield$ be a differential extension field of $\field$.
The solutions in $\exfield$ of the system of differential equations $\{F_1=0, \ldots, F_M=0\}$, fulfilling the inequations in $\sys$, will be denoted by $\Sol_{\exfield}(\sys)$.

\medskip

A \textit{(differential) Thomas decomposition}~\cite[Algorithm~2.2.56]{robertz2014formal}
of a system~\eqref{eq-simple} is a finite collection of simple (differential) systems $\sys_i$ such that
\begin{equation}\label{eq:partitionsolutions}
\Sol_{\exfield}(\sys) = \dot\bigcup \, \Sol_{\exfield}(\sys_i).
\end{equation}
Simple (differential) systems are formally integrable in the sense that they incorporate
all integrability conditions for formal power series solutions.

The computation of a Thomas decomposition of a differential system can be understood as an iteration of two phases (which can also be interwoven as described in~\cite{bachler2012algorithmic}). 
The first phase achieves simplicity in the algebraic sense by applying Euclidean division with case distinctions so as to ensure the triangular shape and that polynomials have non-vanishing initials and discriminants. 
Splittings of the systems arise from these case distinctions. 
In the second phase, differential pseudo-reductions are applied in order to ensure the remaining conditions in Definition~\ref{de:diffsimple}. 
When non-zero pseudo-remainders are obtained, this process has to be restarted.

In the proof of Theorem~\ref{thm:ThomasDimension}, we will use this approach. 
More details can be found in~\cite{robertz2014formal}.

Implementations of the algebraic and differential Thomas decomposition methods have been
developed by T.~B\"achler and M.~Lange-He\-ger\-mann. The implementation for differential systems
has been incorporated into Maple's standard library since Maple 2018 and was also published
in the Computer Physics Communications library \cite{GERDT2019202}.

\subsection{Algebraic Solutions of First Order Autonomous AODEs}\label{sec-algebraic}

Let $\field$ be an algebraically closed field of characteristic zero. 
Let $\field\Pu$ be the field of formal Puiseux series expanded around any $x_0 \in \field$ or around infinity. 
Since the equations and systems of consideration are invariant under translation of the independent variable, we can assume without loss of generality that the formal Puiseux series are expanded around zero or around infinity such as in~\cite{cano2020algebraic}.

In this subsection we consider a subclass of formal Puiseux series, namely algebraic series.
These are $y(x) \in \field\Pu$ such that there exists a non-zero $Q \in \field[x,y]$ with $Q(x,y(x))=0$.
Note that since the field of formal Puiseux series is algebraically closed, all algebraic solutions can be represented as (formal) Puiseux series. 
In~\cite{aroca2005algebraic} it is stated, without an explicit proof, that if there exists one (non-constant) algebraic solution of a first order autonomous AODE, then all of them are algebraic.
For the convenience of the reader we provide a complete proof of this fact (see Theorem~\ref{THM:AllAlgebraic}) and additionally show that the minimal polynomials are equal up to a shift and the multiplication with a constant (Theorem~\ref{THM:AllMinimalPolynomial}).

\vspace*{2mm}

For $d_x,d_y \in \N$, we say that a formal Puiseux series $y(x) \in \field\Pu$ is \textit{$(d_x,d_y)$-algebraic} if and only if $y(x)$ is algebraic over $\field(x)$ with a minimal polynomial $Q(x,y) \in \field[x,y]$ such that $$\deg(Q,x) \leq d_x,~\deg(Q,y) \leq d_y.$$
Rational functions $y(x)=\frac{f(x)}{g(x)} \in \field(x) \setminus \{0\}$ are $(d_x,1)$-algebraic, where $\gcd(f,g)=1$ and $d_x$ is the maximum of the degrees of $f(x)$ and $g(x)$, since they have the minimal polynomial $$Q(x,y)=g(x)\,y-f(x).$$
Lemma 2.5 in~\cite{aroca2005algebraic} can be rewritten as follows.

\begin{lemma}\label{LEM:DxDyAlgebraic}
Let $d_x,d_y>0$. 
Then there exists an autonomous differential polynomial $G \in \field\{y\}$ of order less or equal to $(d_x+1)(d_y+1)-1$ such that for any $y(x) \in \field\Pu$, $y(x)$ is $(d_x,d_y)$-algebraic over $\field(x)$ if and only if $y(x) \in \Sol_{\field\Pu}(G)$.
\end{lemma}

\begin{theorem}\label{THM:AllAlgebraic}
Let $F \in \field[y,y']$ be an irreducible polynomial. 
Let $y(x) \in \field\Pu$ be a non-constant algebraic solution of the differential equation $F(y,y')=0$.
Then all non-constant formal Puiseux series solutions of $F(y,y')=0$ are algebraic over $\field(x)$.
\end{theorem}
\begin{proof}
Let $Q(x,y) \in \field[x,y]$ be the minimal polynomial of $y(x)$
and let $G \in \field[y,\ldots,y^{(d)}]$, where $d=(\deg(Q,x)+1)(\deg(Q,y)+1)-1$, be as in Lemma~\ref{LEM:DxDyAlgebraic}. 
Then we can compute the differential pseudo-remainder of $G$ with respect to $F$
\begin{align} \label{eq-help3}
R(y,y')=I(y)^{k_1}\,S(y,y')^{k_2}\,G-\sum_{0 \leq i \leq d-1} G_i\,F^{(i)},
\end{align}
where $I=\init(F) \in \field[y], S=\sepa(F) \in \field[y,y'], G_i \in \field\{y\}, k_1,k_2 \in \N$ and $\deg(R,y')<\deg(F,y')$. 
By plugging $y(x)$ into equation~\eqref{eq-help3}, we obtain that $R(y(x),y'(x))=0.$
Let us show that $R$ is the zero-polynomial. 
Assume that $R$ has positive degree in $y'$, otherwise $R$ depends only on $y$ and the statement follows. 
Consider the resultant
\begin{align*}
H(y)=\Res(F,R,y')=A_1(y,y')\,F(y,y')+A_2(y,y')\,R(y,y')
\end{align*}
for some $A_1,A_2 \in \field[y,y']$ with $\deg(A_1,y')<\deg(R,y'), \deg(A_2,y')<\deg(F,y').$
Since $H(y(x))=0$ and $y(x)$ is non-constant, $H$ is the zero polynomial. 
Because $F$ is an irreducible polynomial, $F$ divides $A_2$ or $R$, which is a contradiction to the degree conditions. 
Hence, $R(y,y')=0$. 
Now let $\bar y(x)$ be a non-constant Puiseux series solution of the differential equation $F(y,y')=0$.
Substituting $\bar y(x)$ into equation~\eqref{eq-help3}, $$I(\bar y(x))^{k_1}\,S(\bar y(x),\bar y'(x))^{k_2}\,G(\bar y(x),\ldots,\bar y^{(d)}(x))=0.$$
Similar as before, since $\bar y(x)$ is non-constant and $F$ is irreducible, it follows that the first two factors are non-zero and $\bar y(x)$ is a solution of $G=0$.
Then the statement follows from Lemma~\ref{LEM:DxDyAlgebraic}.
\end{proof}

\begin{lemma}\label{lemma:AllMinimalPolynomial}
Let $F \in \field[y,y']$ be an irreducible polynomial. 
Let $y_1(x),y_2(x)$ be non-constant algebraic solutions of the differential equation $F(y,y')=0$ and let $Q_1,Q_2 \in \field[x,y]$ be their minimal polynomials.
Then there is $c \in \field$ such that $$Q_1(x,y)= Q_2(x+c,y).$$
\end{lemma}
\begin{proof}
From Theorem 3.4 in~\cite{aroca2005algebraic} we know that $$d=\deg(Q_i,x)=\deg(F,y').$$
For the proof we need $y_0 \in \field$ such that all curve points $\qq=(y_0,p_0) \in \V_\field(F)$ and $(x_0,y_0) \in \V_\field(Q_i)$, for $i \in \{1,2\}$, are somehow generic points.
To be precise, let $y_0 \in \field$ be such that
the number of different roots of $F(y_0,z) \in \field[z]$ and of $Q_1(x,y_0), Q_2(x,y_0) \in \field[x]$ are maximal, namely equal to $d$. 
Note that by the irreducibility of $F$ and $Q_i$, there are only finitely many exceptional values for $y_0 \in \field$ where this is not fulfilled.
Let $p_1,\ldots,p_d \in \field$, $x_1,\ldots,x_d \in \field$ and $x_1',\ldots,x_d' \in \field$ be the distinct roots of $P(y_0,z)$, $Q_1(x,y_0)$ and $Q_2(x,y_0)$, respectively. 
Centered at $(x_j,y_0) \in \V_\field(Q_1)$, for $j \in \{1,\ldots,d\}$, we can compute by the Newton polygon method for algebraic equations the Puiseux expansions. 
Since every point $(x_j,y_0) \in \V_\field(Q_1)$ is regular, the expansion is unique and a formal power series
$$\varphi_j(x)=y_0+\sum_{k \geq 1} a_{j,k}(x-x_j)^k.$$
Moreover, from~\cite[Lemma 2.4]{aroca2005algebraic} it follows that $\varphi_j(x)$ are solutions of the differential equation $F(y,y')=0$.
In particular, $$F(\varphi_j(x_j),\varphi_j'(x_j))=F(y_0,a_{j,1})=0.$$
By the choice of $y_0$, $\{a_{1,1},\ldots,a_{d,1}\}=\{p_1,\ldots,p_d\}.$
Similarly, for $Q_2$ and its Puiseux expansions $$\psi_j(x)=y_0+\sum_{k \geq 1} a'_{j,k}(x-x_j')^k$$ we obtain $\{a'_{1,1},\ldots,a'_{d,1}\}=\{p_1,\ldots,p_d\}.$
Without loss of generality we can assume that $a_{1,1}=p_1=a'_{1,1}$. 
Since $F$ is independent of $x$ and Cauchy-Kovalevskaya's Theorem is applicable, because $(y_0,p_1)$ is a regular curve point with $\tfrac{\partial F}{\partial z}(y_0,p_1) \ne 0$, all coefficients $a_{1,k}=a'_{1,k}$ coincide and $$y(x)=y_0+\sum_{k \ge 1} a_{1,k}\,x^k$$ is the unique solution of the differential equation $F(y,y')=0$ with the initial condition $(y(0),y'(0))=(y_0,p_1)$ centered around the origin.
Now the irreducible polynomials $\bar Q_1(x,y)=Q_1(x+x_1,y)$ and $\bar Q_2(x,y) = Q_2(x+x_1',y)$ have the common root $y(x)$, which is only possible if they divide each other. 
Since minimal polynomials are monic, the statement follows for $c=x_1'-x_1$.
\end{proof}

\begin{theorem}\label{THM:AllMinimalPolynomial}
Let $F \in \field[y,y']$ be irreducible and let $y(x)$ be a non-constant algebraic solution of the differential equation $F(y,y')=0$ with minimal polynomial $Q \in \field[x,y]$.
Then all non-constant formal Puiseux series solutions of the differential equation $F(y,y')=0$ are algebraic and given by $Q(x+c,y)$, where $c \in \field$.
\end{theorem}
\begin{proof}
First let us prove that, for $c \in \field$, all roots of $Q(x+c,y)$ are indeed solutions of the differential equation $F(y,y')=0$. 
The pseudo-remainder of $F$ with respect to $Q$ is
\begin{equation*}
R(x,y) = I(x)^{k_1}\,S(x,y)^{k_2}\,F(y,y') - A_1\,Q(x,y) - A_2\,Q'(x,y,y')
\end{equation*}
where $I=\init(Q) \in \field[x], S=\sepa(Q) \in \field[x,y], A_1,A_2 \in \field[x,y,y'], k_1,k_2 \in \N$ and $\deg(R,y)<\deg(Q,y)$. 
Since $Q$ is irreducible with the root $y(x)$, $R$ is the zero-polynomial (compare to the proof of Theorem~\ref{THM:AllAlgebraic}). 
In the above equation, substitute $x$ by $x+c$ and set $\bar Q(x,y)=Q(x+c,y)$. 
Note that $\bar Q'(x,y,y')=Q'(x+c,y,y')$ and the initial $\bar I$ and the separant $\bar S$ of $\bar Q$ are $I(x+c)$ and $S(x+c,y)$, respectively. 
Then
\begin{equation*}
0 = \bar I(x)^{k_1}\,\bar S(x,y)^{k_2}\,F(y,y') - \bar A_1\,\bar Q - \bar A_2\,\bar Q',
\end{equation*}
where $\bar A_1(x,y,y')=A_1(x+c,y,y'), \bar A_2(x,y,y')=A_2(x+c,y,y')$. 
Let $\bar y(x)$ be a non-constant root of $\bar Q$. 
Hence, $\bar Q'(x,\bar y(x), \bar y'(x))=0$, $\bar I(x)\,\bar S(x,\bar y(x)) \ne 0$ and consequently, $F(\bar y(x),\bar y'(x)) = 0.$
The converse direction is Lemma~\ref{lemma:AllMinimalPolynomial}.
\end{proof}

\section{Algebraic-Geometric Dimension of Differential Systems}\label{sec-algebraicGeometricDimension}
Let $\field$ be an algebraically closed field of characteristic zero, considered as differential field of constants under the derivation $'$, 
and let $\exfield$ be a differential extension field of $\field$. 
For $n \ge 1$ we denote by $\exfield\{ y_1, \ldots, y_n \}$ the differential polynomial ring in the indeterminates $y_1,\ldots,y_n$. 
We consider systems of algebraic ordinary differential equations
in several differential indeterminates of algebraic dimension 
equal to one in the following sense.

\begin{definition}\label{def-dimension}
Let $F_1,\ldots, F_M$, $U_1,\ldots,U_N \in \exfield[y_1,\ldots,y_1^{(m_1)},\ldots,y_n,\ldots,y_n^{(m_n)}]$ be differential polynomials effectively depending on $y_1^{(m_1)},\ldots,y_n^{(m_n)}$ for some $m_1$, \ldots, $m_n \in \N$. 
Then we define the \textit{(al\-ge\-bra\-ic) dimension} of the corresponding system
\begin{equation}\label{prelim_EQ-AODESystem}
\sys= \{ F_1 = 0, \ldots, F_M = 0 , U_1 \ne 0,\ldots, U_N \ne 0 \}
\end{equation}
as the dimension of $\Va_{\overline{\exfield}}(\sys) \subset \overline{\exfield}^{m_1+\cdots+m_n+n}$ (see Subsection~\ref{subsec:Thomas} for the precise definition of $\Va_{\overline{\exfield}}(\sys)$).
\end{definition}

In what follows we often let $\exfield = K(x)$ with $' = \textrm{d}/\textrm{d}x$.
In that case we would not consider systems~\eqref{prelim_EQ-AODESystem} 
with equations or inequations involving only $x$.

\begin{example}
The differential systems $\{ y_1'^2+y_2^3=0, 2y_1-y_1'y_2=0, y_1 \ne 0 \}$ and $\{ y'+x=0 \}$ have al\-ge\-bra\-ic dimension one, whereas the system
$\{ y'+y=0, y'' \ne 0 \}$ has al\-ge\-bra\-ic dimension two.
\end{example}

Let us remark that $\Va_{\overline{\exfield}}(\sys)$ may be a strict subset of $\Va_{\overline{\exfield}}(\{F_1,\ldots,F_M\})$.

Note that our definition of dimension is purely algebraic and does not correspond to notions of differential dimension
presented, e.g., in~\cite{ritt1950differential} or~\cite{hegemannPhD}. 
In particular, it is not necessary to consider the differential ideal generated by~$\sys$ nor its set of generic solutions. 
As Example~\ref{ex-proj} shows, there are some systems which have algebraic dimension equal to one, but their differential dimension is bigger.

\begin{remark}\label{rem:Dimension}
There are several methods for computing the dimension of an algebraic set, which makes it easy to verify an assumption on the dimension imposed on a system~\eqref{prelim_EQ-AODESystem}. 
For example, for a given system $\sys$ as in~\eqref{prelim_EQ-AODESystem}, an \textit{algebraic} Thomas decomposition $\sys_1,\ldots,\sys_r$ of~$\sys$ 
can be computed (cf., e.g.,~\cite[Subsection~2.2.1]{robertz2014formal}). 
The dimension of $\Va_{\overline{\exfield}}(\sys_i)$ is the difference of the dimension of the ambient affine space
and the number of (distinct) leading variables of the equations in $\sys_i$. 
Then the dimension of~$\Va_{\overline{\exfield}}(\sys)$ is the maximum of the dimensions of the $\Va_{\overline{\exfield}}(\sys_i)$.
Note that each system is a finite set of equations and inequations, so that only finitely
many indeterminates occur, and for our purposes, the coordinate ring of the
ambient affine space is of the form $\field[ y_1,\ldots,y_1^{(m_1)},\ldots,y_n,\ldots,y_n^{(m_n)}]$ for certain $m_1$, \ldots, $m_n \in \N$.
\end{remark}

We are going to explain, in the following theorem, how the algebraic dimension is affected by the construction of a \textit{differential} Thomas decomposition. 
The resulting simple differential systems may refer to different ambient affine spaces in general.

\begin{theorem}\label{thm:ThomasDimension}
Let $\sys$ be a differential system as in~\eqref{prelim_EQ-AODESystem}, of al\-ge\-bra\-ic dimension $d$. 
Then there exists a Thomas decomposition of~$\sys$ (with respect to $<$)
with completely reduced equations and inequations
whose simple differential systems all
have al\-ge\-bra\-ic dimension less than or equal to $d$.
\end{theorem}

\begin{proof}
We recall that a Thomas decomposition of a \textit{differential} system $\sys$ (consisting of
ordinary differential polynomials) can be obtained as follows.
We repeat (if necessary) two stages, namely the computation of an algebraic Thomas decomposition and
differential pseudo-reductions, in a loop.
In the first stage, iterated Euclidean pseudo-reduction is applied to pairs of left hand sides of equations and
inequations of the system, in order to obtain a system that is simple in the algebraic sense (cf.\ Definition~\ref{de:algsimple}).
In general this process requires case distinctions so as to ensure non-vanishing (on the solution
set of the system) of initials and discriminants of the polynomials involved, which leads to a splitting of the system into
subsystems, whose solution sets form a partition of the solution set of the original system.
Let $\sys_1$, \ldots, $\sys_r$ be the resulting (sub)systems, which are simple in the algebraic sense.
Due to the above mentioned partition of the solution set of $\sys$,
if $\sys$ has al\-ge\-bra\-ic dimension $d$, then each $\sys_i$
has al\-ge\-bra\-ic dimension at most $d$ (in the same ambient affine space as for $\sys$).

For each system $\sys_i$ obtained so far, the second stage applies differential pseudo-reductions modulo the
equations in $\sys_i$ as explained in Subsection~\ref{subsec:Thomas}, so as to work towards a simple differential
system with completely reduced equations and inequations. Pseudo-reductions with respect to the greatest variable
(with respect to $<$) are performed first.
In this step of our reasoning we may consider reductions modulo \textit{proper} derivatives of equations $G = 0$ in $\sys_i$
only. Let $\indets(\sys_i)$ be the set of coordinates of the ambient affine space for $\sys_i$, say,
$\indets(\sys_i) = \{ y_1,\ldots,y_1^{(m_1)},\ldots,y_n,\ldots,y_n^{(m_n)} \}$,
and let $\leadvars^{=}(\sys_i)$ be the set of leading variables of equations in $\sys_i$.
The pseudo-remainder $R$ of a reduction of $F$ modulo $G$, where $F = 0$ and $G = 0$ are
equations in $\sys_i$, is either zero or not.

If $R$ is zero, then $F = 0$ is omitted from $\sys_i$. 
The leading variable of $F$ may occur in equations of $\sys_i$ whose leading variables are ranked higher than $\lv(F)$ and possibly in inequations. 
Differential pseudo-reductions modulo $G = 0$ eliminate $\lv(F)$ from $\sys_i$ altogether, because any proper derivative of $G$ has degree one in its leading variable. 
We call the resulting system $\tilde{\sys}_i$. 
The cardinality $|\leadvars^{=}(\tilde{\sys}_i)|$ is $|\leadvars^{=}(\sys_i)|-1$ and $|\indets(\tilde{\sys}_i)|<|\indets(\sys_i)|$. 
Hence, the al\-ge\-bra\-ic dimension of $\tilde{\sys}_i$ is the same or smaller as that of $\sys_i$.

If $R$ is a non-zero element of $\exfield$, then the system $\sys_i$ is inconsistent (as differential system)
and is discarded.

Otherwise $R$ is a differential polynomial with a leading variable $\lv(R)$ that is ranked lower
than $\lv(F)$.
After elimination of $\lv(F)$ from $\sys_i$ (using $G$) as described above and replacing $F=0$ by $R=0$, we obtain a new system~$\tilde{\sys}_i$ with $|\indets(\tilde{\sys}_i)|<|\indets(\sys_i)|$.
The indeterminate $\lv(R)$ is certainly an element of $\indets(\sys_i)$,
but either $\lv(R)$ is an element of $\leadvars^{=}(\sys_i)$ or not.
If not, then $|\leadvars^{=}(\tilde{\sys}_i)|=|\leadvars^{=}(\sys_i)|$ so that
the al\-ge\-bra\-ic dimension of the system drops.
If $\lv(R)$ is an element of $\leadvars^{=}(\sys_i)$, then the
al\-ge\-bra\-ic dimension either does not change or drops.

For differential reductions of inequations $F \neq 0$ modulo (\textit{proper} derivatives of)
equations $G = 0$ in $\sys_i$ we have a similar case distinction.
If the pseudo-remainder $R$ is zero, then the system
is inconsistent and is discarded. If $R$ is a non-zero element of $\exfield$, then the
inequation $F \neq 0$ is omitted from $\sys_i$.
Otherwise $R$ is a differential polynomial with a leading variable $\lv(R)$ that is ranked lower
than $\lv(F)$. In both of these cases, after elimination of further occurrences of $\lv(F)$
in $\sys_i$, for the resulting system~$\tilde{\sys}_i$ it holds that $|\indets(\tilde{\sys}_i)|<|\indets(\sys_i)|$. As a consequence, the
al\-ge\-bra\-ic dimension of the system drops at least by one as well, because
$\leadvars^{=}(\sys_i)=\leadvars^{=}(\tilde{\sys}_i)$.

If any differential reduction was performed on a system $\sys_i$ we start over
with the modified system $\tilde{\sys}_i$ and compute an algebraic Thomas decomposition of it, where the ambient affine space may now be different to the one for the
original system $\sys$. This step may lead to further splittings. Differential
pseudo-reduction might have to be performed again, etc.

This loop is a special case for ordinary differential polynomials
of the Thomas Algorithm and it terminates with finitely many simple differential
systems after finitely many steps.
\end{proof}

\begin{example}\label{ex-proj}
Let $\sys=\{ F_1=z'^2+z = 0 , \ F_2=yz'= 0 \}$ be a system of autonomous AODEs, 
defined over $\field = \C$. Let us check that $\sys$ has algebraic dimension one. 
For that we compute an algebraic Thomas decomposition of~$\sys$ with respect to $y<z<z'$:
\begin{align*}
\tilde{\sys}_1 = \{ y=0, z'^2+z = 0 \}, \quad \tilde{\sys}_2 = \{ z=0, z'=0, y \ne 0 \},
\end{align*}
which are of dimension one because the set of leading variables has two elements. 
A differential Thomas decomposition is given by
\begin{align*}
\sys_1 = \{ y=0, z'^2+z = 0, z \ne 0  \}, \ \sys_2 = \{ y=0, z=0 \}, \ \sys_3 = \{ z=0, y \ne 0 \}.
\end{align*}
Note that $\sys_3$ is of dimension two in $\C[y,z,z']$, but this does not contradict Theorem~\ref{thm:ThomasDimension} because $\sys_3$ is, seen as independent system, of algebraic dimension one.
\end{example}

In the following we show that, if the given differential system is of algebraic dimension one, the construction of a differential Thomas decomposition results in simple systems of a very particular form. 
This leads to results on convergence and on the computation of algebraic solutions (cf. Section~\ref{sec-systems} and Section~\ref{sec-systemalgebraic}.)

\begin{proposition}\label{prop:thomas1dim}
Let $\sys$ be a differential system as in~\eqref{prelim_EQ-AODESystem}
and let the ranking $<$ on $\exfield\{ y_1, \ldots, y_n \}$ be defined as in~\eqref{eq:ranking}.
If $\sys$ is simple (with respect to $<$) and of al\-ge\-bra\-ic dimension
at most one, then it is of one of the following types:
\begin{equation}\tag{I}\label{type1:thomas1dim}
\left\{ \quad
\begin{array}{rcll}
G_s(y_1, \ldots, y_s) & = & 0, & \quad s \in \{ 1, \ldots, t-1 \},\\[0.2em]
G_t(y_1, \ldots, y_t, y'_t) & = & 0,\\[0.2em]
G_s(y_1, \ldots, y_t,y_t',y_{t+1},\ldots,y_s) & = & 0, & \quad s \in \{ t+1, \ldots, n \},\\[0.2em]
U(y_1, \ldots, y_t) & \neq & 0,
\end{array}\right.
\end{equation}
for a unique $t \in \{ 1, \ldots, n \}$, where $\lv(G_t)=y_t'$;
\begin{equation}\tag{II}\label{type2:thomas1dim}
\left\{ \quad
\begin{array}{rcll}
G_s(y_1, \ldots, y_s) & = & 0, & \quad s \in \{ 1, \ldots, n \} \setminus \{ t \},\\[0.2em]
U(y_1, \ldots, y_t) & \neq & 0,
\end{array}\right.
\end{equation}
for a unique $t \in \{ 1, \ldots, n \}$;
\begin{equation}\tag{III}\label{type3:thomas1dim}
\left\{ \quad
\begin{array}{rcll}
G_s(y_1, \ldots, y_s) & = & 0, & \quad s \in \{ 1, \ldots, n \},
\end{array}\right.
\end{equation}
where, in each of the systems~\eqref{type1:thomas1dim}, \eqref{type2:thomas1dim}, \eqref{type3:thomas1dim},
for all $s$ in the admissible range, we have $G_s \in \exfield[y_1, \ldots, y_s]$ with
either $\lv(G_s) = y_s$ or $G_s$ is the zero polynomial and all other equations in the system
are independent of $y_s$, and $U$ is a non-zero element of $\exfield$ or $U \in \exfield[y_1,\ldots,y_t]$ with $\lv(U_t)=y_t$.
\end{proposition}

\begin{proof}
Since $\sys$ is
of al\-ge\-bra\-ic dimension at most one,
there exist a minimal subset $\{ z_1, \ldots, z_q \}$
of $\{ y_1, \ldots, y_n \}$, say, of cardinality $q$, and non-negative integers $m_1$, \ldots, $m_q$
such that all equations and inequations
of $\sys$ are elements
of $\exfield[z_1, \ldots, z_1^{(m_1)}, \ldots, z_q, \ldots, z_q^{(m_q)}] \setminus \exfield$
and $\Va_{\overline{\exfield}}(\sys) \subset \overline{\exfield}^{m_1+\cdots+m_q+q}$ is of dimension at most one.

Considered as an algebraic system, $\sys$ admits solutions
\[
(\eta_{1,0}, \ldots, \eta_{1,m_1}, \eta_{2,0}, \ldots, \eta_{2,m_2}, \ldots,
\eta_{q,0}, \ldots, \eta_{q,m_q}) \in \overline{\exfield}^{m_1+ \cdots + m_q+q},
\]
where the coordinates of the affine space are arranged in accordance with the ranking~\eqref{eq:ranking}.
By simplicity of $\sys$, every solution is obtained as follows (cf.\ also
\cite[Remark~2.2.5]{robertz2014formal}). Suppose that all coordinates of a solution
preceding $\eta_{k,l}$ are already determined. If there is an equation $G_j = 0$ in $\sys$
with $\lv(G_j) = z_k^{(l)}$ and degree $d$ in $\lv(G_j)$, then there are exactly $d$
different $\eta_{k,l} \in \overline{\exfield}$ such that $G_j(\eta_{1,0}, \ldots, \eta_{k,l}) = 0$,
because the initial and discriminant of the polynomial $G_j$ in $\lv(G_j)$ do not
vanish when evaluated at the preceding coordinates, and $\overline{\exfield}$ is algebraically closed.
Similarly, if there is an inequation $G_j \neq 0$ in $\sys$
with $\lv(G_j) = z_k^{(l)}$ and degree $d$ in $\lv(G_j)$, then all $\eta_{k,l} \in \overline{\exfield}$
except $d$ many satisfy $G_j(\eta_{1,0}, \ldots, \eta_{k,l}) \neq 0$.
If there is neither an equation nor an inequation in $\sys$
with leading variable $z_k^{(l)}$, then $\eta_{k,l} \in \overline{\exfield}$ can be chosen arbitrarily.

We enumerate the differential polynomials $G_j$ occurring 
as left hand sides of equations and inequations in $\sys$
in increasing order with respect to their pairwise distinct leading variables $\lv(G_j)$. 
Since in $\sys$ each equation is differentially
reduced modulo the other equations and each inequation is differentially reduced
modulo the equations, for every $k \in \{ 1, \ldots, q \}$ there exists at most
one non-negative integer $\overline{m}_k \le m_k$ such that $z_k^{(\overline{m}_k)}$ is
the leading variable of some equation $G_j = 0$, and if $\overline{m}_k$ exists
and $z_k^{(n_k)}$ is the leading variable of an inequation in $\sys$, then $n_k < \overline{m}_k$.

Suppose that all $\overline{m}_1$, \ldots, $\overline{m}_q$ exist,
i.e., for all $k \in \{ 1, \ldots, q \}$
some equation in $\sys$ has leading variable $z_k^{(\overline{m}_k)}$.
If $\overline{m}_1 = m_1$, \ldots, $\overline{m}_q = m_q$, then, since $\sys$
is of al\-ge\-bra\-ic dimension at most one, either $m_1 = \ldots = m_q = 0$,
in which case $\sys$ is a system of type~\eqref{type3:thomas1dim},
or $m_j = 1$ for a unique $j \in \{ 1, \ldots, q \}$ and $m_k = 0$ for all $k \neq j$,
in which case $\sys$ is a system of type~\eqref{type1:thomas1dim} with $t = j$.
On the other hand, if $\overline{m}_j < m_j$ for some $j \in \{ 1, \ldots, q \}$,
then no equation in $\sys$ involves $z_j^{(m_j)}$, and since $\sys$ is
of al\-ge\-bra\-ic dimension at most one and all
$\overline{m}_1$, \ldots, $\overline{m}_q$ exist, $\sys$ contains no inequation
and is of type~\eqref{type3:thomas1dim}.

Now suppose that $\overline{m}_j$ does not exist.
According to the recursive solution procedure recalled above,
the coordinates $\eta_{j,0}$, $\eta_{j,1}$, \ldots, $\eta_{j,m_j}$ attain
arbitrary values in $\exfield$, possibly with the exception of finitely many values
if $\sys$ contains an inequation with leading variable $z_j^{(n_j)}$.
Since $\sys$ is of al\-ge\-bra\-ic dimension at most one,
we conclude that at most one $\overline{m}_j$ may not exist, and in that case
we have $m_j = 0$, and $\sys$ is of type~\eqref{type2:thomas1dim} with $t = j$.
\end{proof}

Note that a differential system of type~\eqref{type1:thomas1dim}
or~\eqref{type2:thomas1dim} or~\eqref{type3:thomas1dim} in Proposition~\ref{prop:thomas1dim} is of algebraic dimension $1$,$1$ or $0$, respectively. 
In Example~\ref{ex-proj}, the systems~$\sys_1, \sys_2$ and~$\sys_3$ are of type~\eqref{type1:thomas1dim}, \eqref{type3:thomas1dim} and~\eqref{type2:thomas1dim}, respectively.

\begin{corollary}\label{cor:ThomasDim1}
Let $\sys$ be a differential system as in~\eqref{prelim_EQ-AODESystem}
and let the ranking $<$ on $\exfield\{ y_1, \ldots, y_n \}$ be defined as in~\eqref{eq:ranking}.
If $\sys$ is of al\-ge\-bra\-ic dimension one, then there exists
a Thomas decomposition of $\sys$ (with respect to $<$) each of whose simple differential systems
is of one of the types~\eqref{type1:thomas1dim}, \eqref{type2:thomas1dim} and \eqref{type3:thomas1dim} from Proposition~\ref{prop:thomas1dim}.
\end{corollary}
\begin{proof}
This is a combination of Theorem~\ref{thm:ThomasDimension} and
Proposition~\ref{prop:thomas1dim}.
\end{proof}

\section{Solutions of Differential Systems of Dimension One}\label{sec-systems}
Let~$\sys$ be a differential system as in~\eqref{prelim_EQ-AODESystem} with coefficients in the algebraically closed field $\field$, i.e.\ the differential polynomials in~$\sys$ have constant coefficients. 
We are interested in formal Puiseux series solution vectors
\[
Y(x)=(y_1(x),\ldots,y_n(x)) \in \field\Pu^n
\]
of~$\sys$, expanded around $x_0 \in \{ 0, \infty \}$.

\begin{remark}\label{rem-NegativeOrder}
The components of the solution $Y(x)=(y_1(x),\ldots,y_n(x)) \in \field\Pu^n$ can be assumed to have non-negative order.
Otherwise, choose $I \subseteq \{1,\ldots,n\}$ as the set of indices where the order is negative and perform the change of variable $\tilde{y}_i=1/y_i$ in $\sys$ for every $i \in I$.
More precisely, let $\{ y_1,\ldots,y_1^{(m_1)},\ldots,y_n,\ldots,y_n^{(m_n)} \}$ be the ambient space for~$\sys$ and let $m=m_1+\cdots+m_n+n$. 
If the differential polynomials $F_1,\ldots,F_M \in \sys$ have some $y_1,\ldots,y_n$ as factors, divide these equations by such factors. 
Then we define for every $i \in I$ the mapping
$$\Phi_i:\field^{m_i+1}\setminus \V_{\field}(u_0) \rightarrow \field^{m_i+1} \setminus \V_{\field}(w_0); (u_0,\ldots,u_{m_i})\mapsto (w_0,\ldots,w_{m_i}),$$ where $w_0=1/u_0$ and
$$w_j= \dfrac{-u_{0}^{j-1}\,u_j+R_j(u_0,\ldots,u_{j-1})}{u_{0}^{j+1}}$$
for some polynomial $R_j$.
Since the equality above is linear in $u_j$, $\Phi_i$ is birational.
For all $i \notin I$ let us define $\Phi_i$ as the identity map on $\field^{m_i+1}$.
Then the mapping $$\boldsymbol{\Phi}(u_{1,0},\ldots,u_{n,m_n})=(\Phi_1(u_{1,0},\ldots,u_{1,m_1}),\ldots,\Phi_n(u_{n,0},\ldots,u_{n,m_n}))$$ is birational on $\field^m \setminus \bigcup_{i \in I}\V_{\field}(y_i)$.
Let us apply $\boldsymbol{\Phi}$ to $\sys$ and call the resulting system $\sys^*$.
The Zariski closure of $\boldsymbol{\Phi}(\Va_{\field}(\sys))$ is the Zariski closure of $\Va_{\field}(\sys^*)$. 
In particular, $\dim(\Va_{\field}(\sys))=\dim(\Va_{\field}(\sys^*))$ and one may proceed with $\sys^*$ instead of $\sys$.
\end{remark}

In the following we will again impose that the given differential system is of algebraic dimension one. 
This allows to use in particular Proposition~\ref{prop:thomas1dim}.

\begin{remark}\label{rem-solutionDecomposition}
Clearly, by definition of Thomas decomposition (cf.~\eqref{eq:partitionsolutions}),
the set of solution vectors of $\sys$ is partitioned into the subsets of solution vectors of the simple differential systems in Proposition~\ref{prop:thomas1dim}. 
Systems of type~\eqref{type2:thomas1dim} or~\eqref{type3:thomas1dim} are algebraic systems.
Of main interest for our further study of differential systems of dimension are systems of type~\eqref{type1:thomas1dim}.

Following the discussion in Proposition~\ref{prop:thomas1dim}, the construction of a differential Thomas decomposition of~$\sys$ may eliminate certain differential indeterminates $y_s$ in some subsystems~$\sys_i$, see e.g. Example~\ref{ex-proj}. 
Then the solutions of~$\sys_i$ do not include the components corresponding to the $y_s$ and when prolonging to a solution of~$\sys$, these components can be chosen arbitrarily. 
In the following we may speak in this case about the \textit{free variables} $y_s$. 
Additionally, for systems of type~\eqref{type2:thomas1dim}, $y_t$ can be chosen arbitrarily except for a finite number of constants possibly excluded by the inequation $U(y_1,\ldots,y_t) \ne 0$. 
In this situation we call $y_t$ a \textit{parametric variable}.

In systems~$\sys_i$ of the type~\eqref{type3:thomas1dim}, the components which are not free variables are constants. 
The same holds for the first $1 \le s <t$ components of systems~$\sys_i$ of the type~\eqref{type2:thomas1dim}. 
The parametric variable $y_t$ can be chosen (almost) arbitrary. 
The following components depend on this choice and, for example, if $y_t$ and free variables are chosen constant, every of the remaining components is again constant or a free variable.  
In systems~$\sys_i$ is of the type~\eqref{type1:thomas1dim}, the first $1 \le s < t$ components of a solution are constants $\eta_1,\ldots,\eta_{t-1} \in \overline{\exfield}$ or a free variable, which might be chosen constant. 
Plug these values into~$\sys$ in order to obtain a system of type~\eqref{type1:thomas1dim} where $G_t(\eta_1,\ldots,\eta_{t-1},y_t,y_t') \in \overline{\exfield}[y_t,y_t']$ is the polynomial with lowest ranked leading variable. 
Hence, we may assume occasionally that $t=1$.

For showing that the components of the solutions are convergent or algebraic, we may replace the components of the free variables in a solution by zero or any other suitable function. 
A parametric variable $y_t$ can be substituted by almost every constant and every non-constant suitable function such as $x$. 
In this case, the next components $y_{t+1}(x),y_{t+2}(x),\ldots$ have to be substituted accordingly.
\end{remark}

\begin{theorem}\label{theorem-convergence2}
Let $\field=\C$. 
Let $\sys$ be a differential system as is~\eqref{prelim_EQ-AODESystem}, of algebraic dimension one, and let $Y(x)=(y_1(x),\ldots,y_n(x)) \in \C\Pu^n$ be a solution of~$\sys$. 
Let $N$ be a non-negative integer. 
Then there exists a solution $\bar Y(x)=(\bar y_1(x),\ldots,\bar y_n(x))$ of~$\sys$ with convergent components such that $Y(x) \equiv \bar Y(x) \mod x^N$.
\end{theorem}
\begin{proof}
By Corollary~\ref{cor:ThomasDim1}, we may restrict our attention to simple systems $\bar{\sys}$ of types~\eqref{type1:thomas1dim}, \eqref{type2:thomas1dim} and~\eqref{type3:thomas1dim}. 
If $Y(x)$ is a solution of a system of type~\eqref{type2:thomas1dim} or~\eqref{type3:thomas1dim}, the result follows by the Artin approximation theorem (see~\cite[Theorem 1.2]{artin1968solutions}).

In the case that~$\bar{\sys}$ is of type~\eqref{type1:thomas1dim}, the components in $Y(x)$ corresponding to the free variables will be substituted by its truncations up to order $N$. 
For $1 \le s < t$, the component $y_s(x)$ is either a constant or corresponds to a free variable. 
Plugging them into $G_t$, we obtain the autonomous differential equation $H_t(y_t,y_t')=G_t(y_1(x),\ldots,y_{t-1}(x),y_t,y_t')=0$ with leading variable $y_t'$. 
Applying~\cite[Theorem 11]{cano2019existence} to $H_t(y_t,y_t')=0$, the component $y_t(x)$ is convergent.
For $s>t$, let us proof by induction that $y_s(x)$ is convergent. 
If $y_s$ is a free variable, then $y_s(x)$ is convergent. 
Otherwise, let us consider
$$H_s(x,y_s)=G_s(y_1(x),\ldots,y_t(x),y_t'(x),\ldots,y_{s-1}(x),y_s)=0.$$
Since the system is simple, the leading variable of~$H_s$ is~$y_s$ and,
by Puiseux's Theorem, $y_s(x)$ is convergent.
\end{proof}

In the proof of Theorem~\ref{theorem-convergence2} we have evaluated the polynomials $G_s$ at Puiseux series in order to obtain the equations $H_s(x,y_s)=0$. 
In case that $y_t(x)$ is algebraic over $\field(x)$, these evaluations can be performed by using its minimal polynomial which leads to algorithmic computations as we show in the following section.

\begin{theorem}\label{thm:existence}
Let $\sys$ be a differential system as is~\eqref{prelim_EQ-AODESystem}, of algebraic dimension one. 
Then it can be decided algorithmically whether~$\sys$ has a formal Puiseux series solution. 
Moreover, it can be decided whether a formal Puiseux series solution with a non-constant component exists.
\end{theorem}
\begin{proof}
By Corollary~\ref{cor:ThomasDim1}, there exists a Thomas decomposition of~$\sys$ into simple systems of the type~\eqref{type1:thomas1dim}, \eqref{type2:thomas1dim} or~\eqref{type3:thomas1dim} with the same solution set as~$\sys$. 
The third system has constant solutions. 
There are no other formal Puiseux series solutions if and only if it has no free variables (see Remark~\ref{rem-solutionDecomposition}). 
In systems of type~\eqref{type2:thomas1dim}, for the first $1 \le s < t$ equations there exist constant solutions $\eta_1,\ldots,\eta_{t-1}$. 
Then any non-constant formal Puiseux series $y_t(x)$ fulfills the inequality $U \ne 0$. 
Since the field of formal Puiseux series is algebraically closed, the tuple $(\eta_1,\ldots,\eta_{t-1},y_t(x))$ can be prolonged to a non-constant solution (cf.\ proof of Theorem~\ref{theorem-convergence2}).

In the first $1 \le s < t$ equations of systems of the type~\eqref{type1:thomas1dim}, we again obtain constant solution components $\eta_1,\ldots,\eta_{t-1}$. 
The differential equation $G_t = 0$ has infinitely many non-constant formal power series solutions $y_t(x)$ (see e.g.~\cite{cano2019existence}). 
Take one of them and again prolong $(\eta_1,\ldots,\eta_{t-1},y_t(x))$ to a non-constant solution.

The system~$\sys$ does not have any solution if and only if all subsystems are inconsistent and discarded in the process of computing the Thomas decomposition. 
Moreover,~$\sys$ has only constant solutions if and only if all subsystems in the Thomas decomposition are of type~\eqref{type3:thomas1dim} and do not involve free variables.
\end{proof}

\subsection{Algebraic Solutions}\label{sec-systemalgebraic}
The Thomas decomposition of a given differential system~$\sys$ leads to simple subsystems~$\sys_i$ with a disjoint solution set. 
In the case where~$\sys$ is of algebraic dimension one, the~$\sys_i$ can be assumed to be of a specific type (see Proposition~\ref{prop:thomas1dim}). 
In this section, after some possible further decomposition of the~$\sys_i$ and using the results from Section~\ref{sec-algebraic},
we show that the existence of algebraic solutions can be decided and, in the affirmative case, all of them can be determined. 
The algebraic solutions are expressed either as a simple system (involving no derivatives) or by their minimal polynomials.
In the first case, all of the solutions can be recovered whereas in the latter case there might be combinations of roots which are not a solution of the given system~$\sys$.

\begin{lemma}\label{lem:algsol}
Let $\sys$ be a simple system as in~\eqref{type1:thomas1dim} with $t=1$, 
let $P_1 \in \field(x)[y_1]$ be a polynomial with $\lv(P_1)=y_1$ 
and let $<$ be as in~\eqref{eq:ranking}. 
Then the simple differential systems in any Thomas decomposition of $\sys \cup \{P_1=0\}$ with respect to $<$ are of the following type:
\begin{equation}\tag{IV}\label{type:thomas1alg}
\left\{ \quad
\begin{array}{rcll}
H_s(x,y_1, \ldots, y_s) & = & 0, & \quad s \in \{ 1, \ldots, n \}
\end{array}\right.
\end{equation}
where $H_s \in \field(x)[y_1, \ldots, y_s]$ with either $\lv(H_s) = y_s$ or $H_s$ is the zero-polynomial 
and all other polynomials are independent of $y_s$.
\end{lemma}
\begin{proof}
Since $\lv(P_1)=y_1$ is the lowest ranked variable and $P_1$ is already completely reduced modulo~$\sys$, the system $\sys \cup \{P_1=0\}$ is either inconsistent or has algebraic dimension zero. 
From now on we assume the latter.

Let $\sys_1,\ldots,\sys_r$ be the simple subsystems of a Thomas decomposition of $\sys \cup \{P_1=0\}$ with respect to $<$. 
During the construction of the Thomas decomposition, the polynomial $G_1$ from~\eqref{type1:thomas1dim} was reduced by $P_1'$, which has degree one in $y_1'$, and the derivative is eliminated.

Fix $i \in \{1,\ldots,r\}$ and let $\{z_1,\ldots,z_q\}$ be the minimal subset of $\{y_1,\ldots,y_n\}$ such that the elements in~$\sys_i$ are in $\field(x)[z_1,\ldots,z_q]$. 
Since the algebraic dimension of~$\sys_i$ is zero, for every $s \in \{1,\ldots,q\}$ there exists a unique equation $H_s=0$ with $H_s \in \field(x)[z_1,\ldots,z_s]$ and $\lv(H_s)=z_s$. 
Since the set of leading variables of~$\sys_i$ is $\{z_1,\ldots,z_q\}$, the system contains only (algebraic) equations and~$\sys_i$ is of the type~\eqref{type:thomas1alg}.
\end{proof}

Let us note that in Lemma~\ref{lem:algsol} we do not necessarily choose $P_1$ to be irreducible. 

\begin{corollary}\label{cor:simplesolutionsarealgebraic}
Let $\sys$ be a differential system of the type~\eqref{type1:thomas1dim} and~$\mathcal{H}$ be an algebraic subsystem of~$\sys$ of the form~\eqref{type:thomas1alg}. 
Let $Y(x)=(y_1(x),\ldots,y_n(x)) \in \field\Pu^n$ be a solution of~$\mathcal{H}$. 
Then, for every $1 \le s \le n$, the component $y_s(x)$ is algebraic or $y_s$ is a free variable.
\end{corollary}
\begin{proof}
Replace in $Y(x)$ all the components which are not appearing in $\mathcal{H}$ by zero and call the resulting vector $\tilde{Y}(x)=(\tilde{y}_1(x),\ldots,\tilde{y}_n(x))$, which is a solution of $\mathcal{H}$, and therefore of~$\sys$. 
Since $\mathcal{H}$ is simple and of algebraic dimension zero, the tower of field extensions $$\field \subseteq \field(\tilde{y}_1(x)) \subseteq \cdots \subseteq \field(\tilde{y}_1(x),\ldots,\tilde{y}_n(x))$$ is algebraic.
\end{proof}

Due to Corollary~\ref{cor:simplesolutionsarealgebraic}, we may speak in the following about algebraic solutions instead of formal Puiseux series solutions where all components are algebraic. 
Let us now show that all of the algebraic solutions can be found algorithmically.

\begin{lemma}\label{lem:algebraicsolutions}
Let $Y(x)=(y_1(x),\ldots,y_n(x)) \in \field\Pu^n$ be an algebraic solution of a system~$\sys$ as in~\eqref{type1:thomas1dim}. 
Then there exists a simple algebraic subsystem~$\mathcal{H}$ of~$\sys$, of the form~\eqref{type:thomas1alg}, having~$Y(x)$ as solution.
\end{lemma}
\begin{proof}
Let us assume that $t=1$ (cf.\ Remark~\ref{rem-solutionDecomposition}).
The algebraic Puiseux series $y_1(x)$ is a solution of the differential equation $G_1(y,y')=0$ from $\sys$. 
There is a factor $F_1(y,y')$ of the polynomial $G_1$ such that $y_1(x)$ is a solution of the differential equation $F_1(y_1,y_1')=0$. 
Let $Q_1(x,y_1) \in \field[x,y_1]$ be the minimal polynomial of $y_1(x)$. 
Let $\mathcal{H}$ be the algebraic system obtained when 
applying Lemma~\ref{lem:algsol} to $\sys \cup \{Q_1=0\}$. 
Then, $\mathcal{H}$ is among the subsystems derived from the Thomas decomposition of $\sys \cup \{Q_1=0\}$, 
and hence $\mathcal{H}$ is a simple algebraic subsystem of~$\sys$ of the type~\eqref{type:thomas1alg} having $Y(x)$ as solution.
\end{proof}

For details regarding the computation of $Q_1$ figuring in the previous proof we refer to Section~\ref{sec-algebraic}.

\begin{lemma}\label{lem:allalgebraicsolutions}
Let $\mathcal{H}=\{H_1(x,y_1),\ldots,H_n(x,y_1,\ldots,y_n)\}$ be a simple algebraic subsystem of a differential system~$\sys$ of type~\eqref{type1:thomas1dim}. 
Then, for $c \in \field$, every solution of the system~$\mathcal{H}(c)$, defined by
\begin{equation*}
\left\{ \quad
\bar H_s(x,y_1,\ldots,y_s) = H_s(x+c,y_1, \ldots, y_s) = 0, \quad s \in \{ 1, \ldots, n \} \right.,
\end{equation*}
is an algebraic solution of~$\sys$.
\end{lemma}
\begin{proof}
The system~$\mathcal{H}(c)$ is of the type~\eqref{type:thomas1alg}. 
By Corollary~\ref{cor:simplesolutionsarealgebraic}, its solutions are algebraic. 
We need to show that the solutions of $\mathcal{H}(c)$ are solutions of~$\sys$.
Let $P_1 \in \field(x)[y_1]$ be the polynomial used for obtaining $\mathcal{H}$ (see Lemma~\ref{lem:algsol}). 
We may assume that $P_1$ is square-free. 
Since $y_1$ is the lowest ranked variable, the initial of $P_1$ does not vanish on the solution set of~$\sys$. 
Define $\bar P_1(x,y_1) = P_1(x+c,y_1)$ where $c$ is a new variable with $c'=0$. 
The reduction steps in the process of obtaining~$\mathcal{H}$ in a Thomas decomposition of~$\sys \cup \{P_1 = 0 \}$ are as follows:
\begin{enumerate}
	\item Eliminate $y_1'$ in $G_1(y_1,y_1')$ by $P_1'$ resulting into the polynomial $R_1(x,y_1)$.
	\item Compute the greatest common divisor of $R_1$ and $P_1$, say $T_1$.
	\item Possibly reduce $G_2,\ldots,G_n$ by $T_1$.
\end{enumerate}
The chain rule implies that elimination of $y_1'$ in $G_1(y_1,y_1')$ by $\bar P_1'$ yields $R_1(x+c,y_1)$. 
Similarly, as $x+c$ can be considered as a new variable, the greatest common divisor of $R_1(x+c,y_1)$ and $\bar P_1$ is $T_1(x+c,y_1)$. 
Hence, this process results in the simple system~$\mathcal{H}(c)$ which can be obtained from~$\mathcal{H}$ by substituting $x$ by $x+c$.
\end{proof}

\begin{theorem}\label{thm:algSolDecomposition}
Let $\sys$ be a differential system as in~\eqref{prelim_EQ-AODESystem}, of algebraic dimension one. 
Then there exist a finite number of simple algebraic subsystems~$\mathcal{H}_k$ of~$\sys$ such that any algebraic solution of~$\sys$ is among the solutions of some~$\mathcal{H}_k(c)$, $c \in K$, which are defined as in Lemma~\ref{lem:allalgebraicsolutions}.
\end{theorem}
\begin{proof}
Let us first construct a set of simple algebraic subsystems $\mathcal{H}_k$ of~$\sys$. 
By Corollary~\ref{cor:ThomasDim1}, $\sys$ can be decomposed into simple systems of the form~\eqref{type1:thomas1dim},\eqref{type2:thomas1dim} and~\eqref{type3:thomas1dim}. 
Systems of the second and third type are already algebraic systems and we add them.

Let us consider systems~$\bar{\sys}$ of the type~\eqref{type1:thomas1dim} with $t=1$. 
By possibly further splitting systems according to the factorization of $G_1(y_1,y_1')$, we may additionally assume that $G_1$ in~$\bar\sys$ is irreducible. 
The differential equation $G_1=0$ might have a non-constant algebraic solution, given by its minimal polynomial $Q_1 \in K[x,y_1]$, or does not have any. 
In the first case, let $\mathcal{H}_k$ be the simple algebraic subsystems obtained by a Thomas decomposition of $\bar\sys \cup \{Q_1 = 0\}$.
In the latter case, we discard $\bar\sys$. 
Additionally, compute the constant solutions $y_1(x)=c$ of $G_1$ and let $\mathcal{H}_k$ be the simple algebraic subsystems obtained by a Thomas decomposition of $\bar\sys \cup \{ y_1-c=0 \}$.

Let~$\bar{\sys}$ be of the type~\eqref{type1:thomas1dim} with $t>1$. 
For $1 \le s<t$, compute the constant solutions $y_s(x)=c_s$ and plug them into~$\bar{\sys}$. 
Now proceed with the differential equation $G_t(y_t,y_t') = 0$ as above in order to obtain a non-trivial minimal polynomial $Q_t(x,y_t)$ of a non-constant algebraic solution, if it exists, and the corresponding simple algebraic subsystems $\mathcal{H}_k$. 

Let $Y(x)=(y_1(x),\ldots,y_n(x))$ be an algebraic solution of~$\sys$ and hence, of a subsystem $\bar\sys$. 
If all components of $Y(x)$ are constant, then it is a constant solution of some system~$\bar{\sys}$ and all of them are kept in~$\mathcal{H}_k$.
So let $y_t(x)$ be the first non-constant component. 
Then, by Theorem~\ref{THM:AllMinimalPolynomial}, there exists $c \in K$ such that $Q_t(x+c,y_t(x))=0$ for the minimal polynomial $Q_t$ corresponding to~$\bar\sys$. 
By Lemma~\ref{lem:allalgebraicsolutions}, the simple algebraic systems of $\bar{\sys} \cup \{ Q_t(x+c,y_1) = 0 \}$ are $\mathcal{H}_k(-c)$ and exactly one of them has $Y(x)$ as solution.
\end{proof}

Based on Corollary~\ref{cor:simplesolutionsarealgebraic}, an algebraic solution $(y_1(x),\ldots,y_n(x))$ of a given system~$\sys$ as in~\eqref{type1:thomas1dim} could be alternatively represented by a list of minimal polynomials $(Q_1(x,y_1),\ldots,Q_n(x,y_n))$, called a \textit{minimal polynomial system} of~$\sys$. 
The following proposition gives a degree bound on the minimal polynomials.

\begin{proposition}\label{prop:minpolynomialsystem}
Let $\sys$ be a differential system as in~\eqref{type1:thomas1dim} with an algebraic solution $(y_1(x),\ldots,y_n(x))$. 
Then there exists a minimal polynomial system $(Q_1(x,y_1),\ldots,Q_n(x,y_n))$ of~$\sys$. 
Moreover, every $Q_s \in \field[x,y_s]$ fulfills
\begin{equation}\label{eq-degBound}
\deg_{y_s}(Q_s) \leq (\deg_{y_1}(G_{1})+\deg_{y_1'}(G_{1}))\,\deg_{y_2}(G_{2})\cdots \deg_{y_n}(G_{n}).
\end{equation}
\end{proposition}
\begin{proof}
Without loss of generality we can assume that $t=1$ in~$\sys$ (cf. Remark~\ref{rem-solutionDecomposition}). 
By the primitive element theorem, there exist minimal polynomials $Q_1 \in \field[x,y_1]$, \ldots, $Q_n \in \field[x,y_n]$ of $y_1(x), \ldots, y_n(x)$, respectively.
Using~\cite[Theorem 6]{cano2020algebraic}, we have for $Q_1(x,y_1)$ the degree bound $$\deg_{y_1}(Q_1) \leq \deg_{y_1}(G_{1})+\deg_{y_1'}(G_{1}).$$
Applying the multiplicative formula for the degree to the tower of field extensions, the statement follows.
\end{proof}

Let us note that, in contrast to the situation for simple subsystems of the type~\eqref{type:thomas1alg}, not every combination of roots of the minimal polynomial system indeed defines a solution as it can be seen in Example~\ref{example}.

\section{Algorithm and Examples}\label{sec-Algorithms}
In~\cite[Algorithm 3.6]{bachler2012algorithmic} an algorithm to derive a Thomas decomposition of a given differential system is presented.
We refer to this algorithm by \textsf{ThomasDecomposition}.
In Section 4 of~\cite{aroca2005algebraic} there is a description of an algorithm that decides whether a given first order autonomous AODE $F(y,y')=0$ with an irreducible polynomial $F$ has algebraic solutions and compute them in the affirmative case.
Let us call this algorithm \textsf{AlgebraicSolve} and let the output be equal to the minimal polynomial of an algebraic solution, if it exists, or empty otherwise.

The next algorithm decides whether a system~$\sys$ as in~\eqref{prelim_EQ-AODESystem}, with algebraic dimension 1, has algebraic solutions and describes all of them in the affirmative case.

\begin{algorithm}[H]
\caption{SimpleSystemSolve} \label{ALG:SimpleSystem}
\begin{algorithmic}[1]
	\REQUIRE A system~$\sys$ of algebraic dimension one as in~\eqref{prelim_EQ-AODESystem}.
	\ENSURE A (finite) union of simple algebraic systems of the form~\eqref{type:thomas1alg} or~\eqref{type2:thomas1dim} with the same algebraic solutions as~$\sys$. 
	\STATE Apply \textsf{ThomasDecomposition} (with respect to the ordering~\eqref{eq:ranking}) to~$\sys$ and let $\{ \sys_k \}$ be the simple subsystems.
	\STATE Add simple systems of the type~\eqref{type2:thomas1dim} and~\eqref{type3:thomas1dim} to the output.
	\FOR{every system~$\sys_k$ of the form~\eqref{type1:thomas1dim}}
		\STATE Compute the constant solutions of the first $1 \le s < t$ many components, which are not free variables, and plug them into the remaining equations. 
		\STATE Compute a factorization of $G_t(y_t,y_t')$ and decompose~$\sys_k$ into subsystems where $G_t$ is replaced by the respective factor.
		\FOR{every such subsystem of $\sys_k$}
			\STATE Check by \textsf{AlgebraicSolve} whether $G_t=0$ has an algebraic solution.
			\STATE In the affirmative case, let $Q_t \in K[x,y_t]$ be the minimal polynomial of such an algebraic solution. 
			Apply \textsf{ThomasDecomposition} (with respect to the ordering~\eqref{eq:ranking}) to~$\sys_k \cup \{Q_t = 0\}$ in order to obtain the simple subsystems of~$\sys$ of type~\eqref{type:thomas1alg} and add them to the output.
		\ENDFOR
	\ENDFOR
\end{algorithmic}
\end{algorithm}

\begin{theorem}
Algorithm~\ref{ALG:SimpleSystem} is correct.
\end{theorem}
\begin{proof}
Using Proposition~\ref{prop:thomas1dim}, the formal Puiseux series solutions of~$\sys$ is the disjoint union of the solutions of the~$\sys_k$. 
Systems~$\sys_k$ of the type~\eqref{type2:thomas1dim} and~\eqref{type3:thomas1dim} lead to constant solution components, possibly depending on a parametric variable, and components which are free variables. 
Systems of the type~\eqref{type3:thomas1dim} are in particular of the form~\eqref{type:thomas1alg}.
The non-constant solutions are derived from~$\sys_k$ of type~\eqref{type1:thomas1dim}. 
If such a system has an algebraic solution vector, then the algebraic simple subsystems are computed and cover all algebraic solutions (see Theorem~\ref{thm:algSolDecomposition}).

Then termination follows from the termination of the sub-algorithms and the finite representation of their outputs.
\end{proof}

Let us illustrate the previous ideas and results in the following examples.

\begin{example}\label{ex-1-part3}
Let us consider the system of differential equations of algebraic dimension equal to one given by
\begin{equation}
\sys=
\left\{ \begin{array}{ll}
yy'y''+y'^3-yy''-y'^2=0 \\
z^3-2y'^2+yy'-1=0 \\
z^3+yy''-y'^2=0 \\
3z^2z'-4y'y''=0
\end{array} \right.
\end{equation}
The system~$\sys$ has a differential Thomas decomposition (with respect to the ordering $y<y'<z<z'$) into the single simple system
\begin{equation*}
\sys_1=
\left\{ \begin{array}{ll}
y^2z^3-2=0 \\
yy'-1=0 \\
y \ne 0
\end{array} \right.
\end{equation*}
Its ambient affine space is $\C[y,y',z]$ and the leading variables of the equations are $z$ and $y'$, respectively. 
Hence, the algebraic dimension of~$\sys_1$ is again one.
The equation $yy'-1=0$ has an algebraic solution with minimal polynomial $Q_1(x,y)=y^2-2x.$
This leads to the algebraic solutions given by the simple system
$$\{ Q_1(x,y)=y^2-2x, \ Q_2(x,z)=xz^3-1 \},$$
which are already the minimal polynomials of the solution components.
\end{example}

\begin{example}\label{example}
Let us consider the system
\begin{align*}
\sys= \left\{ \begin{array}{ll}
G_1 = 8y'^3-27y=0 \\
G_2 = z^5-y^3=0 \\
G_3 = 5z^4z'-3y^2y'=0
\end{array}\right.
\end{align*}
It is already a simple system with ambient affine space $\C[y,y',z,z']$ and the equations have leading variables $y',z$ and $z'$, respectively. 
Hence, $\sys$ is indeed of algebraic dimension one.

By direct computation, the solutions of $G_1(y,y')=0$ are $y_1(x)=x^{3/2}$, $y_2(x)=-x^{3/2}$, implicitly defined by $Q_1(x,y)=y^2-x^3.$
From $G_2(y,z)=0$ we see that $z_1(x)=\zeta\,x^{9/10}$, $z_2(x)=-\zeta\,x^{9/10}$ with $\zeta^{5}=1$. 
By plugging them into $G_3=0$, we obtain that $(y_1(x),z_1(x))$, $(y_2(x),z_2(x))$ are solutions of~$\sys$, but neither $(y_1(x),z_2(x))$ nor $(y_2(x),z_1(x))$.

Now let us use Algorithm~\ref{ALG:SimpleSystem}. 
A Thomas decomposition of~$\sys$ (with respect to the ordering $y<y'<z<z'$) is
$$\sys_1=\{y^3-z^5=0, \ 8y'^3-27y=0, \ y \ne 0 \}, \ \sys_2= \{ y=0, \ z=0 \}.$$  
The algebraic dimension of~$\sys_1$ is one and that of~$\sys_2$ is zero.
The equation $8y'^3-27y=0$ has the algebraic solutions given by $Q_1$. 
A Thomas decomposition of $\sys_1 \cup \{Q_1 = 0\}$ is $$\{Q_1(x,y)=y^2-x^3, \ H_2(x,y,z)=z^5-x^3\,y \}.$$
The corresponding minimal polynomial system is $$\{Q_1(x,y)=y^2-x^3, \ Q_2(x,z)=z^{10}-x^{9} \}.$$
In the latter system, not all combinations of roots are indeed solutions of~$\sys_1$ or~$\sys$.
\end{example}

\appendix

\bibliographystyle{acm}

\begin{thebibliography}{10}
	
	\bibitem{aroca2005algebraic}
	{\sc Aroca, J., Cano, J., Feng, R., and Gao, X.-S.}
	\newblock Algebraic {G}eneral {S}olutions of {A}lgebraic {O}rdinary
	{D}ifferential {E}quations.
	\newblock In {\em Proceedings of the 2005 International Symposium on Symbolic
		and Algebraic Computation\/} (2005), ACM, pp.~29--36.
	
	\bibitem{artin1968solutions}
	{\sc Artin, M.}
	\newblock On the solutions of analytic equations.
	\newblock {\em Inventiones mathematicae 5}, 4 (1968), 277--291.
	
	\bibitem{bachler2012algorithmic}
	{\sc B{\"a}chler, T., Gerdt, V., Lange-Hegermann, M., and Robertz, D.}
	\newblock Algorithmic thomas decomposition of algebraic and differential
	systems.
	\newblock {\em Journal of Symbolic Computation 47}, 10 (2012), 1233--1266.
	
	\bibitem{boulier1995representation}
	{\sc Boulier, F., Lazard, D., Ollivier, F., and Petitot, M.}
	\newblock Representation for the radical of a finitely generated differential
	ideal.
	\newblock In {\em Proceedings of the 1995 international symposium on Symbolic
		and algebraic computation\/} (1995), pp.~158--166.
	
	\bibitem{cano2020algebraic}
	{\sc Cano, J., Falkensteiner, S., and Sendra, J.}
	\newblock Algebraic, {R}ational and {P}uiseux {S}eries {S}olutions of {S}ystems
	of {A}utonomous {A}lgebraic {O}{D}{E}s of {D}imension {O}ne.
	\newblock {\em Mathematics in Computer Science\/}
	(2021), 189--198.
	
	\bibitem{cano2019existence}
	{\sc Cano, J., Falkensteiner, S., and Sendra, J.}
	\newblock Existence and {C}onvergence of {P}uiseux {S}eries {S}olutions for
	{F}irst {O}rder {A}utonomous {D}ifferential {E}quations.
	\newblock {\em Journal of Symbolic Computation 108\/} (2022), 137--151.
	
	\bibitem{CluzeauHubert2003}
	{\sc Cluzeau, T., and Hubert, E.}
	\newblock Resolvent {R}epresentation for {R}egular {D}ifferential {I}deals.
	\newblock {\em Applicable Algebra in Engineering, Communication and Computing
		13}, 5 (2003), 395--425.
	
	\bibitem{Denef1984}
	{\sc Denef, J., and Lipshitz, L.}
	\newblock Power {S}eries {S}olutions of {A}lgebraic {D}ifferential {E}quations.
	\newblock {\em Mathematische Annalen 267\/} (1984), 213--238.
	
	\bibitem{GERDT2019202}
	{\sc Gerdt, V.~P., Lange-Hegermann, M., and Robertz, D.}
	\newblock {The MAPLE package TDDS for computing Thomas decompositions of
		systems of nonlinear PDEs}.
	\newblock {\em Computer Physics Communications 234\/} (2019), 202--215.
	
	\bibitem{hegemannPhD}
	{\sc Lange-Hegermann, M.}
	\newblock {\em Counting Solutions of Differential Equations}.
	\newblock PhD thesis, RWTH Aachen University, 2014.
	
	\bibitem{sendra2015rational}
	{\sc Lastra, A., Sendra, J., Ngo, L., and Winkler, F.}
	\newblock Rational {G}eneral {S}olutions of {S}ystems of {A}utonomous
	{O}rdinary {D}ifferential {E}quations of {A}lgebro-{G}eometric {D}imension
	{O}ne.
	\newblock {\em Publicationes Mathematicae Debrecen 86\/} (2015), 49--69.
	
	\bibitem{ritt1950differential}
	{\sc Ritt, J.}
	\newblock {\em Differential {A}lgebra}, vol.~33.
	\newblock American Mathematical Society, 1950.
	
	\bibitem{robertz2014formal}
	{\sc Robertz, D.}
	\newblock {\em Formal algorithmic elimination for PDEs}, vol.~2121.
	\newblock Springer, 2014.
	
\end{thebibliography}

\end{document}